\documentclass[11pt,dvipdfmx]{amsart}

\usepackage{amsmath,amssymb}
\usepackage{amsthm}
\usepackage[abbrev]{amsrefs}
\usepackage{latexsym}
\usepackage{txfonts}
\usepackage{graphicx}
\usepackage{tikz}
\usepackage{bm}

\allowdisplaybreaks[1]

\title[Kummer-type congruences for multi-poly-Bernoulli numbers]{Kummer-type congruences for \\ multi-poly-Bernoulli numbers}
\date{}
\author{Yu Katagiri}

\theoremstyle{definition}
\newtheorem{theorem}{Theorem}[section]
\newtheorem*{theorem*}{Theorem}
\newtheorem{definition}[theorem]{Definition}
\newtheorem*{definition*}{Definition}
\newtheorem{lemma}[theorem]{Lemma}
\newtheorem*{lemma*}{Lemma}

\newtheorem*{proposition*}{Proposition}

\newtheorem*{example*}{Example}
\newtheorem{remark}[theorem]{Remark}
\newtheorem*{remark*}{Remark}

\newtheorem{corollary*}{Corollary}

\newcommand{\st}[2]{\genfrac{\{}{\}}{0pt}{}{#1}{#2}} % Stirling # of the second kind


\begin{document}
\maketitle

\begin{abstract}
The multi-poly-Bernoulli numbers are generalizations of the Bernoulli numbers. In this paper, we will prove Kummer-type congruences for multi-poly-Bernoulli numbers via $p$-adic distributions.
\end{abstract}

\section{Introduction}

For a non-negative integer $n$, the ($n$-th) Bernoulli number $B_n$ is defined by the generating function
\begin{align*}
\frac{te^t}{e^t-1}=\sum_{n=0}^\infty B_n \frac{t^n}{n!}
\end{align*}
as formal power series over $\mathbb{Q}$. It is well known that the following congruence holds (cf. \cite[Theorem 11.6]{AIK}). For positive integers $m, n, N$ and an odd prime $p$,  if $m \equiv n \bmod (p-1)p^{N-1}$, then we have
\begin{align*}
(1-p^{m-1})\frac{B_m}{m} \equiv (1-p^{n-1})\frac{B_n}{n} \bmod p^N.
\end{align*}
This congruence is called the Kummer congruence.

In \cite{Ka97} and \cite{AK99}, Arakawa and Kaneko introduced the poly-Bernoulli numbers $B_n^{(k)}$ and $C_n^{(k)}$, which are generalizations of the Bernoulli numbers, as follows. Let $k$ be an integer and $n$ be a non-negative integer. Poly-Bernoulli numbers $B_n^{(k)}$ and $C_n^{(k)}$ are defined by
\begin{align*}
\frac{\operatorname{Li}_k(1-e^{-t})}{1-e^{-t}}&=\sum_{n=0}^\infty B_n^{(k)} \frac{t^n}{n!}, \\
\frac{\operatorname{Li}_k(1-e^{-t})}{e^t-1}&=\sum_{n=0}^\infty C_n^{(k)} \frac{t^n}{n!}
\end{align*}
respectively, as formal power series over $\mathbb{Q}$. Here, 
\begin{align*}
\operatorname{Li}_k(t)=\sum_{n=1}^\infty \frac{t^n}{n^k}
\end{align*}
is the $k$-th polylogarithm. Note that $\operatorname{Li}_1(t)=-\log(1-t)$ and $B_n^{(1)}=(-1)^nC_n^{(1)}=B_n$ for $n \geq 0$. Kitahara proved the following congruence for poly-Bernoulli numbers by using $p$-adic distributions.

\begin{theorem}[{\cite[Theorem 12]{Ki12}}]\label{Kitahara}
Let $k$ be an integer, $p$ be an odd prime, and $m$, $n$ and $N$ be positive integers with $m, n \geq N$ and $k < p-1$. If $m \equiv n \bmod (p-1)p^{N-1}$, then we have
\begin{align*}
p^{2k'}B_m^{(k)} \equiv p^{2k'}B_n^{(k)} \bmod p^N,
\end{align*}
where $k'={\rm max}\{k, 0\}$.
\end{theorem}

\begin{remark}
Sakata gave an elementary proof of Theorem \ref{Kitahara} in the case $k<0$ (\cite[Theorem 6.1]{Sa14}).
\end{remark}

In this paper, we will consider a further generalization of Theorem \ref{Kitahara}.

\begin{definition}[{\cite[Section 1]{IKT}}]
For $\textbf{k}=(k_1, \cdots , k_r) \in \mathbb{Z}^r$, define the multiple polylogarithm to be
\begin{align*}
\operatorname{Li}_{\textbf{k}}(t)=\sum_{0<m_1<\cdots<m_r}\frac{t^{m_r}}{m_1^{k_1}\cdots m_r^{k_r}}.
\end{align*}
Multi-poly-Bernoulli numbers $B_n^{(\textbf{k})}$ and $C_n^{(\textbf{k})}$ are defined to be the rational numbers satisfying
\begin{align*}
\frac{\operatorname{Li}_{\textbf{k}}(1-e^{-t})}{1-e^{-t}}&=\sum_{n=0}^\infty B_n^{(\textbf{k})} \frac{t^n}{n!}, \\
\frac{\operatorname{Li}_{\textbf{k}}(1-e^{-t})}{e^t-1}&=\sum_{n=0}^\infty C_n^{(\textbf{k})} \frac{t^n}{n!}
\end{align*}
respectively, as formal power series over $\mathbb{Q}$.
\end{definition}

\begin{remark}\label{BandC}
In \cite{IKT}, some relations between $B_n^{(\textbf{k})}$ and $C_n^{(\textbf{k})}$ were proved. For examples, we have relations
\begin{align*}
B_n^{(\textbf{k})}&=\sum_{i=0}^n \dbinom{n}{i}C_i^{(\textbf{k})}, \\
C_n^{(\textbf{k})}&=\sum_{i=0}^n (-1)^{n-i}\dbinom{n}{i}B_i^{(\textbf{k})}, \\
B_n^{(\textbf{k})}&=C_n^{(\textbf{k})}+C_{n-1}^{(k_1, k_2, \cdots, k_r-1)}
\end{align*}
for any $r \geq 1$, $\textbf{k}=(k_1, k_2, \cdots, k_r) \in \mathbb{Z}^r$ and $n \geq 1$ (\cite[Section 2]{IKT}).
\end{remark}

\begin{remark}
The multiple polylogarithm was introduced in \cite{AK99}. It is expected to have relations with the multiple zeta values and the multiple zeta functions. It is also known that the multi-poly-Bernoulli numbers $C_n^{(\textbf{k})}$ describe the finite multiple zeta values (\cite[Theorem 8]{IKT}).
\end{remark}

We call $\textbf{k}=(k_1, \cdots , k_r) \in \mathbb{Z}^r$ an index. For an index $\textbf{k}$, we define the weight of $\textbf{k}$ to be $\operatorname{wt}(\textbf{k})=k_1+\cdots +k_r$ and write $k_i'={\rm max} \{ k_i,0 \}$ and $\textbf{k}^+=(k_1', \cdots , k_r')$. We will prove the following result in Section 3.

\begin{theorem}\label{main}
Let $\textbf{k} \in \mathbb{Z}^r$ be an index, $p$ be an odd prime and $m$, $n$ and $N$ be positive integers with $m, n \geq N$ and $\operatorname{wt}(\textbf{k}^+) < p-1$. If $m \equiv n \bmod (p-1)p^{N-1}$, then we have
\begin{align*}
p^{2\operatorname{wt}(\textbf{k}^+)}B_m^{(\textbf{k})} &\equiv p^{2\operatorname{wt}(\textbf{k}^+)}B_n^{(\textbf{k})} \bmod p^N, \\
p^{2\operatorname{wt}(\textbf{k}^+)}C_m^{(\textbf{k})} &\equiv p^{2\operatorname{wt}(\textbf{k}^+)}C_n^{(\textbf{k})} \bmod p^N.
\end{align*}
\end{theorem}

In Section 4, we will consider the multi-poly-Bernoulli-star numbers, which were introduced in \cite{Im14}, and find Kummer-type congruences for the multi-poly-Bernoulli-star numbers which are similar to Theorem \ref{main}.

\vspace{10pt}
\noindent
\textsc{Notation:}~ In this paper, let $p$ be a prime. For $x \in \mathbb{Q}_p$, we denote the $p$-adic valuation by $\operatorname{ord}_p(x)$. For a real number $x$, $\lfloor x \rfloor$ means the greatest integer less than or equal to $x$.

\vspace{10pt}
\noindent
\textsc{Acknowledgment:}~ The author is grateful to my supervisor Professor Takao Yamazaki for his advice and helpful comments. The author would like to thank Shinichi Kobayashi and Yasuo Ohno for their helpful comments. The author would also like to thank Naho Kawasaki for informing me about previous works. The author also thanks Masato Uchimagi for reading the manuscript carefully. The author was supported by the WISE Program for AI Electronics, Tohoku University.

\section{Preliminaries}

In this section, we will recall a theory of $p$-adic distributions.

\begin{definition}
Let $h$ be a non-negative integer. Define $\operatorname{LA}_h(\mathbb{Z}_p, \mathbb{Q}_p)$ to be the set of functions $f : \mathbb{Z}_p \rightarrow \mathbb{Q}_p$ which is locally analytic at each point with radius of convergence $\ge p^{-h}$. For $f \in \operatorname{LA}_h(\mathbb{Z}_p, \mathbb{Q}_p)$, the norm of $f$ is given by
\begin{align*}
||f||_h={\rm sup}_{n \geq 0, \ a \in \mathbb{Z}_p} \{|p^{nh}a_n|_p\}
\end{align*}
for the expansion $f(x)=\sum_{n=0}^\infty a_n(x-a)^n$ on $a+p^h\mathbb{Z}_p$. The set $\operatorname{LA}_h(\mathbb{Z}_p, \mathbb{Q}_p)$ is a $\mathbb{Q}_p$-vector space equipped with the topology induced by the norm. Since there exist natural inclusions $\operatorname{LA}_h(\mathbb{Z}_p, \mathbb{Q}_p) \rightarrow \operatorname{LA}_{h+1}(\mathbb{Z}_p, \mathbb{Q}_p)$ for all $h \geq 0$, we may define $\operatorname{LA}(\mathbb{Z}_p, \mathbb{Q}_p)=\cup_{h \geq 0} \operatorname{LA}_h(\mathbb{Z}_p, \mathbb{Q}_p)$ equipped with the inductive limit topology. A continuous $\mathbb{Q}_p$-linear map $\mu : \operatorname{LA}(\mathbb{Z}_p, \mathbb{Q}_p) \rightarrow \mathbb{Q}_p$ is called a $p$-adic distribution and we write
\begin{align*}
\int_{\mathbb{Z}_p} f(x) d\mu(x) \coloneqq \mu(f)
\end{align*}
for $f \in \operatorname{LA}(\mathbb{Z}_p, \mathbb{Q}_p)$. We denote by $D(\mathbb{Z}_p)$ the set of $p$-adic distributions.
\end{definition}

It is known that the following theorems hold.

\begin{theorem}[{\cite[Lemma 1]{Ma58}}]\label{Mahler}
Let $f : \mathbb{Z}_p \rightarrow \mathbb{Q}_p$. The function $f$ is continuous if and only if there exist $a_n \in \mathbb{Q}_p$ such that
\begin{align*}
f(x)=\sum_{n=0}^\infty a_n \dbinom{x}{n}
\end{align*}
and $a_n \to 0$ as $n \to \infty$. Here, we define
\begin{align*}
\dbinom{x}{0}=1, \ \dbinom{x}{n}=\frac{x(x-1)\cdots (x-n+1)}{n!} \in \mathbb{Q}[x]
\end{align*}
for $n \geq 1$.
\end{theorem}

\begin{theorem}[{\cite[Th\'eor\`eme 3]{Am64}}]\label{Amice1}
Let $h$ be a non-negative integer. For $f : \mathbb{Z}_p \rightarrow \mathbb{Q}_p$, $f \in \operatorname{LA}_h(\mathbb{Z}_p, \mathbb{Q}_p)$ if and only if there exist $a_n \in \mathbb{Q}_p$ such that
\begin{align*}
f(x)=\sum_{n=0}^\infty a_n \left\lfloor \frac{n}{p^h}\right\rfloor ! \dbinom{x}{n}
\end{align*}
and $a_n \to 0$ as $n \to \infty$. Moreover, $||f||_h \leq 1$ holds if and only if $a_n \in \mathbb{Z}_p$ for all $n \geq 0$.
\end{theorem}

\begin{theorem}[{\cite[Theorem 2.3]{ST01}}]\label{Amice2}
Let $R$ be the set of formal power series $f(T)$ over $\mathbb{Q}_p$ which converges on the open unit disk. Then the map $D(\mathbb{Z}_p) \rightarrow R$ given by
\begin{align*}
\mu \mapsto \int_{\mathbb{Z}_p} (1+T)^x d\mu(x) \coloneqq \sum_{n=0}^\infty \int_{\mathbb{Z}_p}\dbinom{x}{n}d\mu(x)T^n
\end{align*}
is bijective. The inverse map sends $\sum_{n=0}^\infty c_n T^n \in R$ to the element of $D(\mathbb{Z}_p)$ given by
\begin{align}\label{inv}
\operatorname{LA}(\mathbb{Z}_p, \mathbb{Q}_p) \rightarrow \mathbb{Q}_p ; \ f(x)=\sum_{n=0}^\infty a_n \dbinom{x}{n} \mapsto \sum_{n=0}^\infty a_nc_n.
\end{align}
\end{theorem}

\begin{remark}
Since $f \in \operatorname{LA}(\mathbb{Z}_p, \mathbb{Q}_p)$ is continuous on $\mathbb{Z}_p$, it follows from Theorem \ref{Mahler} that $f$ has the expansion as (\ref{inv}) and the infinite sum in (\ref{inv}) is convergent.
\end{remark}

Note that, if a formal power series $f(T) \in R$ corresponds to a $p$-adic distribution $\mu$, we have
\begin{align*}
\left((1+T)\frac{d}{dT}\right) f(T)=\int_{\mathbb{Z}_p} x(1+T)^x d\mu(x)=\sum_{n=0}^\infty \int_{\mathbb{Z}_p} x \dbinom{x}{n}d\mu(x)T^n
\end{align*}
and
\begin{align}\label{diffprop}
\left. \left((1+T)\frac{d}{dT}\right)^n f(T) \right|_{T=0} = \int_{\mathbb{Z}_p} x^n d\mu (x)
\end{align}
for $n \geq 0$. Indeed, we can check these by using the property
\begin{align*}
x \dbinom{x}{n}=(n+1)\dbinom{x}{n+1}+n\dbinom{x}{n}.
\end{align*}

\section{Proof of Theorem \ref{main}}

In this section, we will prove Theorem \ref{main}. Our proof is inspired by the proof of \cite[Theorem 12]{Ki12}. In the following, let $p$ be an odd prime.

For positive integers $m, n$ and $N$, by applying Theorem \ref{Amice1} to the case $h=1$ and $p^{-N}(x^m-x^n) \in \operatorname{LA}_1(\mathbb{Z}_p, \mathbb{Q}_p)$, we obtain $a_j \in \mathbb{Q}_p$ satisfying
\begin{align*}
\frac{x^m-x^n}{p^N}=\sum_{j=0}^\infty a_j \left\lfloor \frac{j}{p} \right\rfloor ! \dbinom{x}{j}
\end{align*}
and $|a_j|_p \to 0$ as $j \to \infty$. 

\begin{lemma}\label{prelemma}
If $m, n \geq N$ and $m \equiv n \bmod (p-1)p^{N-1}$, then we have $a_j \in \mathbb{Z}_p$ for any $j \geq 0$.
\end{lemma}

\begin{proof}
Put $P(x)=p^{-N}(x^m-x^n)$. According to Theorem \ref{Amice1}, we must prove $||P(x)||_1 \leq 1$ and it suffices to show that $Q(y) \coloneqq P(c+py) \in \mathbb{Z}_p[y]$ for any $c=0, 1, \cdots, p-1$. If $c=0$, it is clear.

Suppose that $c \neq 0$. We put $m-n=(p-1)p^{N-1}d$ with $d \in \mathbb{Z}_{>0}$ and
\begin{align*}
Q(y)&=p^{-N}(c+py)^n\{(c+py)^{(p-1)p^{N-1}d}-1\}.
\end{align*}
We will check that $(c+py)^{(p-1)p^{N-1}d} \equiv 1 \bmod p^N\mathbb{Z}_p[y]$ by induction on $N$. When $N=1$, we see that $(c+py)^{(p-1)d} \equiv c^{(p-1)d} \equiv 1 \bmod p\mathbb{Z}_p[y]$. Let $N >0$ and suppose that the assertion holds for $N$. Then there exists a polynomial $R_N(y) \in \mathbb{Z}_p[y]$ such that $(c+py)^{(p-1)p^{N-1}d}=1+p^NR_N(y)$ and we have
\begin{align*}
(c+py)^{(p-1)p^Nd}&=(1+p^NR_N(y))^p \\
                          &=\sum_{i=0}^p \binom{p}{i} p^{Ni}R_N(y)^i \equiv 1 \bmod p^{N+1}\mathbb{Z}_p[y].
\end{align*}
This completes the proof.
\end{proof}

\begin{proof}[Proof of Theorem \ref{main}]
We omit the proof for $C_n^{(\textbf{k})}$ because it can be checked by the same argument as the following proof for $B_n^{(\textbf{k})}$. Put
\begin{align*}
f(x)=\frac{\operatorname{Li}_{\textbf{k}}(1-e^x)}{1-e^x}
\end{align*}
and $g(T)=f(\log (1+T))$. In other words, we set
\begin{align*}
f(x)&=\sum_{0<m_1<\cdots<m_r}\frac{(1-e^x)^{m_r-1}}{m_1^{k_1}\cdots m_r^{k_r}} = \sum_{n=0}^\infty (-1)^nB_n^{(\textbf{k})} \frac{x^n}{n!}, \\
g(T)&=\sum_{0<m_1<\cdots<m_r} \frac{(-1)^{m_r-1}}{m_1^{k_1}\cdots m_r^{k_r}}T^{m_r-1}.
\end{align*}
We can check that $g(T)$ converges on the open unit disk. Indeed, since we have
\begin{align*}
\left|\sum_{0<m_1<\cdots<m_r} \frac{(-1)^{m_r-1}}{m_1^{k_1}\cdots m_r^{k_r}}\right|_p \leq m_r^{\operatorname{wt}(\textbf{k}^+)},
\end{align*}
it follows that
\begin{align*}
\limsup_{m_r \to \infty} \left|\sum_{0<m_1<\cdots<m_r} \frac{(-1)^{m_r-1}}{m_1^{k_1}\cdots m_r^{k_r}}\right|_p^{\frac{1}{m_r}} =1.
\end{align*}
Using Theorem \ref{Amice2}, we get a $p$-adic distribution $\mu$ corresponding to $g$. The $p$-adic distribution $\mu : \operatorname{LA}(\mathbb{Z}_p,\mathbb{Q}_p) \rightarrow \mathbb{Q}_p$ is given by
\begin{align*}
\varphi \mapsto \sum_{j=r-1}^\infty (-1)^ja_j \sum_{0<m_1<\cdots<m_{r-1}<j+1} \frac{1}{m_1^{k_1}\cdots m_{r-1}^{k_{r-1}}(j+1)^{k_r}},
\end{align*}
where $\varphi$ has the expansion $\varphi(x)=\sum_{j=0}^\infty a_j \binom{x}{j}$. According to (\ref{diffprop}), we obtain that
\begin{align*}
\int_{\mathbb{Z}_p} x^n d\mu(x) &= \left. \left((1+T)\frac{d}{dT}\right)^n g(T) \right|_{T=0} \\
                                           &= \left. \left(\frac{d}{dx}\right)^n f(x) \right|_{x=0} = (-1)^n B_n^{(\textbf{k})}
\end{align*}
for $n \geq 0$.

For positive integers $m, n$ and $N$ with $m \equiv n \bmod (p-1)p^{N-1}$, Theorem \ref{Amice1} implies that there exist $a_j \in \mathbb{Q}_p$ such that
\begin{align*}
\frac{x^m-x^n}{p^N}=\sum_{j=0}^\infty a_j \left\lfloor \frac{j}{p} \right\rfloor ! \dbinom{x}{j}
\end{align*}
and $|a_j|_p \to 0$ as $j \to \infty$. Then we have $a_j \in \mathbb{Z}_p$ for any $j \geq 0$ by Lemma \ref{prelemma}. We see that
\begin{align*}
\int_{\mathbb{Z}_p} \frac{x^m-x^n}{p^N} d\mu(x) &= \sum_{j=0}^\infty a_j \left\lfloor \frac{j}{p} \right\rfloor ! \int_{\mathbb{Z}_p} \dbinom{x}{j} d\mu(x) \\
                                                                   &= \sum_{j=0}^\infty (-1)^ja_j \left\lfloor \frac{j}{p} \right\rfloor ! \sum_{0<m_1<\cdots<m_{r-1}<j+1} \frac{1}{m_1^{k_1}\cdots m_{r-1}^{k_{r-1}}(j+1)^{k_r}}.
\end{align*}
Put
\begin{align}\label{defh}
h(j)=\left\lfloor \frac{j}{p} \right\rfloor ! \sum_{0<m_1<\cdots<m_{r-1}<j+1} \frac{1}{m_1^{k_1}\cdots m_{r-1}^{k_{r-1}}(j+1)^{k_r}}
\end{align}
for $j \geq r-1$. Note that the summation in the R.H.S. of  (\ref{defh}) is empty for $0 \leq j \leq r-2$ and understood to be 0. We will prove the following lemma soon later.

\begin{lemma}\label{keylemma}
If $\operatorname{wt}(\textbf{k}^+) < p-1$, then we have
\begin{align*}
\min_{j \geq r-1} \{ \operatorname{ord}_p(h(j)) \} \geq -2\operatorname{wt}(\textbf{k}^+).
\end{align*}
\end{lemma}

It follows from the above lemma that
\begin{align*}
p^{2\operatorname{wt}(\textbf{k}^+)}\int_{\mathbb{Z}_p} \frac{x^m-x^n}{p^N} d\mu = p^{2\operatorname{wt}(\textbf{k}^+)-N} \left\{(-1)^m B_m^{(\textbf{k})}-(-1)^n B_n^{(\textbf{k})}\right\} \in \mathbb{Z}_p.
\end{align*}
It is equivalent to the congruence
\begin{align*}
p^{2\operatorname{wt}(\textbf{k}^+)}B_m^{(\textbf{k})} \equiv p^{2\operatorname{wt}(\textbf{k}^+)}B_n^{(\textbf{k})} \bmod p^N.
\end{align*}
\end{proof}

We will show Lemma \ref{keylemma}.

\begin{proof}[Proof of Lemma \ref{keylemma}]
Let $\textbf{k}=(k_1, \cdots , k_r)$. For $j \leq p-1$, we see that $\operatorname{ord}_p(h(j))\geq -k_r$. Set $j=ap+i \ (\geq p)$ with $a \geq 1$ and $0 \leq i \leq p-1$. Then we have
\begin{align*}
&\min_{0 \leq i \leq p-1} \left\{ \operatorname{ord}_p(h(ap+i))\right\} \\
=&\min_{0 \leq i \leq p-1}\left\{ \operatorname{ord}_p(a!)-k_r\operatorname{ord}_p(ap+i+1)+\operatorname{ord}_p\left(\sum_{0<m_1<\cdots<m_{r-1}<ap+i+1} \frac{1}{m_1^{k_1}\cdots m_{r-1}^{k_{r-1}}}\right) \right\} \\
\geq &\min_{0 \leq i \leq p-1} \left\{ \operatorname{ord}_p(a!)-k_r'\operatorname{ord}_p(ap+i+1)+\min_{0<m_1<\cdots<m_{r-1}<ap+i+1} \left\{-\sum_{s=1}^{r-1}k_s'\operatorname{ord}_p(m_s)\right\} \right\} \\
=&\operatorname{ord}_p(a!)-k_r'\operatorname{ord}_p(a+1)-\max_{0<m_1<\cdots<m_{r-1}<(a+1)p}\left\{\sum_{s=1}^{r-1}k_s'\operatorname{ord}_p(m_s)\right\}-k_r' \\
\geq & \operatorname{ord}_p(a!)-k_r'\operatorname{ord}_p(a+1)-\max_{0<b_1<\cdots<b_{r-1}\leq a}\left\{\sum_{s=1}^{r-1}k_s'\operatorname{ord}_p(b_s)\right\}-\operatorname{wt}(\textbf{k}^+) \eqqcolon F(a).
\end{align*}
It is enough to prove that $\min_{a \geq 1}\{F(a)\} \geq -2\operatorname{wt}(\textbf{k}^+)$. For $t \geq 0$ and $0 \leq u \leq p-1$, since we see that
\begin{align*}
\operatorname{ord}_p((tp+u)!)=\operatorname{ord}_p((tp+p-1)!)
\end{align*}
and
\begin{align*}
\max_{0<b_1<\cdots<b_{r-1}\leq tp+u}\left\{\sum_{s=1}^{r-1}k_s'\operatorname{ord}_p(b_s)\right\} \leq \max_{0<b_1<\cdots<b_{r-1}\leq tp+p-1}\left\{\sum_{s=1}^{r-1}k_s'\operatorname{ord}_p(b_s)\right\},
\end{align*}
it suffices to check the case $a \equiv p-1 \bmod p$. Putting $a=qp^l-1$ with $l \geq 1$, $q \geq 1$ and $p \nmid q$, we have
\begin{align*}
&F(qp^l-1) \\
=&\operatorname{ord}_p\left(\frac{(qp^l)!}{qp^l}\right)-k_r'\operatorname{ord}_p(qp^l)-\max_{0<b_1<\cdots<b_{r-1}\leq qp^l-1}\left\{\sum_{s=1}^{r-1}k_s'\operatorname{ord}_p(b_s)\right\}-\operatorname{wt}(\textbf{k}^+) \\
=&\operatorname{ord}_p((qp^l)!)-(k_r'+1)\operatorname{ord}_p(qp^l)-\max_{0<b_1<\cdots<b_{r-1}\leq qp^l-1}\left\{\sum_{s=1}^{r-1}k_s'\operatorname{ord}_p(b_s)\right\}-\operatorname{wt}(\textbf{k}^+) \\
=&q\frac{p^l-1}{p-1}+\operatorname{ord}_p(q!)-(k_r'+1)l-\max_{0<b_1<\cdots<b_{r-1}\leq qp^l-1}\left\{\sum_{s=1}^{r-1}k_s'\operatorname{ord}_p(b_s)\right\}-\operatorname{wt}(\textbf{k}^+).
\end{align*}
If $1 \leq q \leq p-1$, since $b_s \leq (p-1)p^l-1<p^{l+1}$ and $\operatorname{ord}_p(b_s) \leq l$ for $1 \leq s \leq r-1$, we find that
\begin{align*}
F(qp^l-1)&=q\frac{p^l-1}{p-1}-(k_r'+1)l-\max_{0<b_1<\cdots<b_{r-1}\leq qp^l-1}\left\{\sum_{s=1}^{r-1}k_s'\operatorname{ord}_p(b_s)\right\}-\operatorname{wt}(\textbf{k}^+) \\
&\geq q\frac{p^l-1}{p-1}-(k_r'+1)l-\left(\sum_{s=1}^{r-1} k_s'\right)l-\operatorname{wt}(\textbf{k}^+) \\
&\geq \frac{p^l-1}{p-1}-(\operatorname{wt}(\textbf{k}^+) +1)l-\operatorname{wt}(\textbf{k}^+) \\
&
\begin{cases}
=-2\operatorname{wt}(\textbf{k}^+) & \text{if} \ l=1 \\
\geq p+1-2(\operatorname{wt}(\textbf{k}^+) +1)-(p-2) & \text{if} \ l \geq 2
\end{cases} \\
& \geq -2\operatorname{wt}(\textbf{k}^+).
\end{align*}
Note that we used the assumption $\operatorname{wt}(\textbf{k}^+)<p-1$ in the case $l \geq 2$. 

If $q \geq p+1$, set $q=\sum_{i=0}^d c_i p^i$ with $0 \leq c_i \leq p-1$, $c_0c_d \neq 0$ and $d \geq 1$. Then it follows that
\begin{align*}
&F(qp^l-1) \\
\geq &\frac{p^l-1}{p-1}\sum_{i=0}^d c_i p^i+\frac{1}{p-1}\sum_{i=1}^d c_i (p^i-1)-(k_r'+1)l-\left(\sum_{s=1}^{r-1} k_s'\right)(d+l)-\operatorname{wt}(\textbf{k}^+) \\
\geq &\frac{p^l-1}{p-1}(p^d+1)+\frac{p^d-1}{p-1}-(\operatorname{wt}(\textbf{k}^+)+1)l-\left(\sum_{s=1}^{r-1} k_s'\right)d-\operatorname{wt}(\textbf{k}^+) \\
=&\frac{p^{l+d}+p^l-2}{p-1}-(\operatorname{wt}(\textbf{k}^+)+1)l-\left(\sum_{s=1}^{r-1} k_s'\right)d-\operatorname{wt}(\textbf{k}^+) \\
\geq &\frac{p^{d+1}+p-2}{p-1}-\left(\sum_{s=1}^{r-1} k_s'\right)d-2\operatorname{wt}(\textbf{k}^+)-1 \\
=&\left(1+\frac{1}{p-1}\right)p^d-\left(\sum_{s=1}^{r-1} k_s'\right)d-2\operatorname{wt}(\textbf{k}^+)-\frac{1}{p-1} \\
\geq &\left(1+\frac{1}{p-1}\right)p-\sum_{s=1}^{r-1} k_s'-2\operatorname{wt}(\textbf{k}^+)-\frac{1}{p-1} \\
=&\left(p-\sum_{s=1}^{r-1} k_s'\right)+1-2\operatorname{wt}(\textbf{k}^+)>-2\operatorname{wt}(\textbf{k}^+).
\end{align*}
This completes the proof.
\end{proof}

\begin{remark}\label{explicit}
We obtain the explicit formula of $B_n^{(\textbf{k})}$ by using the $p$-adic distribution $\mu$ in the proof of Theorem \ref{main} as follows. For $n \geq 0$, it is known that we have
\begin{align*}
x^n=\sum_{j=0}^n \st{n}{j} j! \dbinom{x}{j},
\end{align*}
where, for any integers $a$ and $b$, $\st{a}{b}$ are called the Stirling numbers of the second kind and defined by the recurrence formula
\begin{align*}
\st{a+1}{b}=\st{a}{b-1}+b\st{a}{b}
\end{align*}
with the conditions $\st{0}{0}=1$ and $\st{a}{b}=0$ for $a<b$ (\cite[Definition 2.2, Proposition 2.6]{AIK}). Then we find that
\begin{align*}
B_n^{(\textbf{k})}&=(-1)^n\int_{\mathbb{Z}_p} x^n d\mu(x)=(-1)^n\sum_{j=0}^n \st{n}{j} j! \int_{\mathbb{Z}_p}\dbinom{x}{j} d\mu(x) \\
                      &=(-1)^n\sum_{j=0}^n \st{n}{j} j! \sum_{0<m_1<\cdots<m_{r-1}<j+1} \frac{(-1)^j}{m_1^{k_1}\cdots m_{r-1}^{k_{r-1}}(j+1)^{k_r}} \\
                      &=(-1)^n\sum_{0<m_1<\cdots<m_{r-1}<m_r\leq n+1} \frac{(-1)^{m_r-1}(m_r-1)!\st{n}{m_r-1}}{m_1^{k_1}\cdots m_{r-1}^{k_{r-1}}m_r^{k_r}}.
\end{align*}
By the exactly same way, we get
\begin{align*}
C_n^{(\textbf{k})}=(-1)^n\sum_{0<m_1<\cdots<m_{r-1}<m_r\leq n+1} \frac{(-1)^{m_r-1}(m_r-1)!\st{n+1}{m_r}}{m_1^{k_1}\cdots m_{r-1}^{k_{r-1}}m_r^{k_r}}.
\end{align*}
These formulas were proved in \cite[Theorem 3]{IKT} by using the generating functions.
\end{remark}

\begin{remark}
 It was claimed in \cite[Theorem 13]{Ki12} that, given an odd prime $p$ and positive integers $m, n, k, N$ with $p \ge \max\{k+2, (N+k)/2\}$ and $m \equiv n \bmod (p-1)p^N$, one has $p^kB_m^{(k)} \equiv p^kB_n^{(k)} \bmod p^N$. However, there are counterexamples: $pB_1^{(1)}=p/2 \not\equiv 0=pB_m^{(1)} \bmod p^N$ for $N \ge 2$ and $m=(p-1)p^N+1$. (Its proof breaks down at \cite[Proposition 11]{Ki12}, for which $j=p^2+p-1$ yields a counterexample.)
\end{remark}

\section{Multi-poly-Bernoulli-star numbers}

At the end of this paper, we will give Kummer-type congruences for other Bernoulli numbers.

\begin{definition}[{\cite[Section 1]{Im14}}]
For $\textbf{k}=(k_1, \cdots , k_r) \in \mathbb{Z}^r$, define the non-strict multiple polylogarithm to be
\begin{align*}
\operatorname{Li}_{\textbf{k}}^{\star}(t)=\sum_{0<m_1\leq \cdots \leq m_r}\frac{t^{m_r}}{m_1^{k_1}\cdots m_r^{k_r}}.
\end{align*}
The multi-poly-Bernoulli-star numbers $B_{n, \star}^{(\textbf{k})}$ and $C_{n, \star}^{(\textbf{k})}$ are defined to be the rational numbers satisfying
\begin{align*}
\frac{\operatorname{Li}_{\textbf{k}}^{\star}(1-e^{-t})}{1-e^{-t}}&=\sum_{n=0}^\infty B_{n, \star}^{(\textbf{k})} \frac{t^n}{n!}, \\
\frac{\operatorname{Li}_{\textbf{k}}^{\star}(1-e^{-t})}{e^t-1}&=\sum_{n=0}^\infty C_{n, \star}^{(\textbf{k})} \frac{t^n}{n!}
\end{align*}
respectively, as formal power series over $\mathbb{Q}$.
\end{definition}

\begin{remark}
Similar relations to Remark \ref{BandC} were proved in \cite[Propositions 2.3, 2.4]{Im14}. Furthermore, the multi-poly-Bernoulli-star numbers $B_{n, \star}^{(\textbf{k})}$ and $C_{n, \star}^{(\textbf{k})}$ verify a duality relation for $\textbf{k}=(k_1, \cdots , k_r) \in \mathbb{Z}_{>0}^r$ (\cite[Theorem 3.2]{Im14}).
\end{remark}

\begin{remark}
It is known that the multi-poly-Bernoulli-star numbers $C_{n, \star}^{(\textbf{k})}$ describe finite multiple zeta-star values (\cite[Section 4]{Im14}).
\end{remark}

The following theorem can be shown by the exactly same argument as Theorem \ref{main} and hence is omitted.

\begin{theorem}
Let $\textbf{k} \in \mathbb{Z}^r$ be an index, $p$ be an odd prime and $m$, $n$ and $N$ be positive integers with $m, n \geq N$ and $\operatorname{wt}(\textbf{k}^+) < p-1$. If $m \equiv n \bmod (p-1)p^{N-1}$, then we have
\begin{align*}
p^{2\operatorname{wt}(\textbf{k}^+)}B_{m, \star}^{(\textbf{k})} &\equiv p^{2\operatorname{wt}(\textbf{k}^+)}B_{n, \star}^{(\textbf{k})} \bmod p^N, \\
p^{2\operatorname{wt}(\textbf{k}^+)}C_{m, \star}^{(\textbf{k})} &\equiv p^{2\operatorname{wt}(\textbf{k}^+)}C_{n, \star}^{(\textbf{k})} \bmod p^N.
\end{align*}
\end{theorem}

\begin{remark}
We can check the following formulas
\begin{align*}
B_{n, \star}^{(\textbf{k})}&=(-1)^n\sum_{0<m_1\leq\cdots \leq m_{r-1} \leq m_r\leq n+1} \frac{(-1)^{m_r-1}(m_r-1)!\st{n}{m_r-1}}{m_1^{k_1}\cdots m_{r-1}^{k_{r-1}}m_r^{k_r}}, \\
C_{n, \star}^{(\textbf{k})}&=(-1)^n\sum_{0<m_1\leq \cdots \leq m_{r-1} \leq m_r\leq n+1} \frac{(-1)^{m_r-1}(m_r-1)!\st{n+1}{m_r}}{m_1^{k_1}\cdots m_{r-1}^{k_{r-1}}m_r^{k_r}}
\end{align*}
by the same computation as Remark \ref{explicit}. These were obtained in \cite[Proposition 2.2]{Im14} by using the generating functions.
\end{remark}

\vspace{10pt}
\noindent
Mathematical Institute, Graduate School of Science, Tohoku University,\\
6-3 Aramakiaza, Aoba, Sendai, Miyagi 980-8578, Japan.\\
E-mail address: \textbf{yu.katagiri.s3@dc.tohoku.ac.jp}

\end{document}